\documentclass[10pt]{amsart}

\usepackage{lmodern}
\usepackage[T1]{fontenc}

\usepackage{amsmath, amssymb, amsthm, amscd, euscript, xcolor}
\definecolor{dblue}{rgb}{0,0,0.6}
\usepackage[colorlinks=true, linkcolor=dblue, citecolor=dblue, filecolor = dblue, menucolor = dblue, urlcolor = dblue]{hyperref}
\usepackage{tikz-cd}
\usetikzlibrary{arrows, decorations.pathmorphing}

\def\AA{\mathcal A}

\def\EE{\mathcal E}
\def\FF{\mathcal F}
\def\GG{\mathcal G}
\def\HH{\mathcal H}

\def\KK{\mathcal K}
\def\LL{\mathcal L}
\def\MM{\mathcal M}
\def\OO{\mathcal O}

\def\WW{\mathcal W}
\def\UU{\mathcal U}

\def\Z{\mathbb{Z}}

\newcommand{\bbZ}{\mathbb{Z}}

\newcommand{\bbC}{\mathbb{C}}

\newcommand{\bbP}{\mathbb{P}}

\newcommand{\rk}{\mathrm{rk}\,}
\renewcommand{\le}{\leqslant}
\renewcommand{\ge}{\geqslant}

\newcommand{\Rder}{\mathrm{R}}
\newcommand{\Lder}{\mathrm{L}}
\newcommand{\Ext}{\mathrm{Ext}}
\newcommand{\IExt}{\EuScript{E}xt}
\newcommand{\Hom}{\mathrm{Hom}}
\newcommand{\IHom}{\EuScript{H}om}

\newcommand{\Db}{\EuScript{D}^b}

\newcommand{\Ati}[1]{\mathrm{At}_{#1}}
\newcommand{\Tr}{\mathrm{Tr}}
\newcommand{\CH}{\mathrm{CH}}
\newcommand{\AJ}{\mathrm{AJ}}
\newcommand{\IJac}[1]{\mathrm{J}(#1)}
\newcommand{\Jac}{\mathrm{J}}
\newcommand{\Gr}{\mathrm{Gr}}
\newcommand{\coker}{\mathrm{coker}}

\newcommand{\lrarr}{\longrightarrow}
\newcommand{\hrarr}{\hookrightarrow}
\newcommand{\sdot}{{\raisebox{0.16ex}{$\scriptscriptstyle\bullet$}}}

\newtheorem{defn}{Definition}[section]
\newtheorem{prop}[defn]{Proposition}
\newtheorem{thm}[defn]{Theorem}
\newtheorem{lemma}[defn]{Lemma}

{\theoremstyle{remark}

}

\makeatletter
\@addtoreset{equation}{section}
\makeatother

\title{On the geometry of the Lehn--Lehn--Sorger--van Straten eightfold}

\author{Evgeny Shinder}
\address{The University of Sheffield,
School of Mathematics and Statistics,
The Hicks Building, Hounsfield Rd, Sheffield S3 7RH}

\author{Andrey Soldatenkov}
\address{Mathematisches Institut, Universit\"at Bonn,
Endenicher Allee 60, 53115 Bonn, Germany}


\begin{document}

\begin{abstract}
In this note we make a few remarks about the geometry of the
holomorphic symplectic manifold $Z$ constructed in
\cite{LLSvS} as a two-step contraction of the variety of twisted cubic curves
on a cubic fourfold $Y\subset \bbP^5$. 

We show that $Z$
is birational to a component of the moduli space of stable sheaves 
in the Calabi-Yau subcategory of the derived category of $Y$.
Using this description we deduce that the twisted
cubics contained in a hyperplane section $Y_H = Y \cap H$ of $Y$ give rise to a
Lagrangian subvariety $Z_H \subset Z$. For a generic choice of the hyperplane, 
$Z_H$ is birational to the theta-divisor in the intermediate Jacobian $\IJac{Y_H}$.
\end{abstract}

\maketitle

\tableofcontents

\section{Introduction}

We work over the field of complex numbers. Throughout the paper $Y\subset \bbP^5$ is
a smooth cubic fourfold not containing a plane. In \cite{LLSvS} the variety $M_3(Y)$ of generalized
twisted cubic curves on $Y$ was studied. It was shown that $M_3(Y)$ is 10-dimensional,
smooth and irreducible. Starting from this variety an 8-dimensional irreducible holomorphic symplectic
(IHS) manifold $Z$ was constructed. More precisely, it was shown that there exist morphisms
\begin{equation}\label{eqn_contractions}
M_3(Y)\stackrel{a}{\longrightarrow} Z'\stackrel{\sigma}{\longrightarrow} Z,
\end{equation}
and
\begin{equation}\label{eqn_embedding}
\mu\colon Y\hookrightarrow Z,
\end{equation}
where $a$ is a $\bbP^2$-fibre bundle and $\sigma$ is the blow-up along the image
of $\mu$. 

It was later shown in \cite{AL} that $Z$ is birational --- and hence deformation
equivalent --- to a Hilbert scheme of four points on a K3 surface.



In this paper we present another point of view on $Z$. We show that an open
subset of $Z$ can be described as a moduli space of Gieseker stable torsion-free sheaves of rank $3$ on $Y$.

Kuznetsov and Markushevich \cite{KM} have constructed a closed two-form on any moduli space of sheaves on $Y$.
Properties of the Kuznetsov-Markushevich form are known to be closely related to the structure of the 
derived category of $Y$.
The bounded derived category $\Db(Y)$ of coherent sheaves on $Y$ 
has an exceptional collection $\OO_Y$, $\OO_Y(1)$, $\OO_Y(2)$ with right orthogonal
$\AA_Y$, so that $\Db(Y)=\langle\AA_Y, \OO_Y, \OO_Y(1), \OO_Y(2)\rangle$. The category
$\AA_Y$ is a Calabi-Yau category of dimension two, meaning that its Serre functor is the shift by
$2$ \cite[Section 4]{K}. 

It was shown in \cite{KM} that
the two-form on moduli spaces of sheaves on $Y$ is non-degenerate if the sheaves lie in $\AA_Y$. 
The torsion-free sheaves mentioned above lie in $\AA_Y$.
This gives an alternative description of the symplectic form on $Z$: 

\medskip
{\bf Theorem \ref{thm_MMF}.}
{\it The component $\MM_F$ of the moduli space of Gieseker stable rank 3 sheaves on $Y$ 
with Hilbert polynomial $\frac38 n^4 + \frac94 n^3 + \frac{33}{8} n^2 + \frac94 n$
is birational to the IHS manifold $Z$. Under this birational equivalence 
the symplectic form on $Z$ defined in \cite{LLSvS} corresponds
to the Kuznetsov-Markushevich form on $\MM_F$.}
\medskip


A similar approach relying on the description of an open part of $Z$ as a moduli space
was used by Addington and Lehn in \cite{AL} to prove that the variety $Z$ is a deformation
of a Hilbert scheme of four points on a K3 surface. In \cite{O} Ouchi considered the
case of cubic fourfolds containing a plane. He proved that one can describe (a birational model) of
the LLSVS variety as a moduli space of Bridgeland-stable objects in the derived category
of a twisted K3 surface. Moreover, in this situation one also has a Lagrangian embedding
of the cubic fourfold into the LLSVS variety as in (\ref{eqn_embedding}).

Another similar construction has been proposed by \cite{LMS} who proved that $Z$ is birational
to a component of the moduli space of stable vector bundles of rank $6$ on $Y$.

Using the birational equivalence between $Z$ and the moduli space of sheaves on $Y$
we show that twisted cubics lying in hyperplane sections $Y_H$ of $Y$ give rise
to Lagrangian subvarieties in $Z$ and discuss the geometry of these subvarieties:

\medskip
{\bf Theorem.}
{\it Denote by $Z_H$ the image in $Z$ of twisted cubics lying in a hyperplane section $Y_H = Y\cap H$
under the map $a$ from (\ref{eqn_contractions}). If $Y$ and $H$ are generic, then $Z_H$ is a Lagrangian subvariety of $Z$
which is birational to the theta-divisor of the intermediate Jacobian of $Y_H$.}
\begin{proof} See Proposition \ref{prop_Lagr} and Theorem \ref{thm_AJ}.
\end{proof}

This is analogous to the case of lines on $Y$: it is well-known that lines on $Y$ form 
an IHS fourfold, and lines contained in hyperplane sections of $Y$ form Lagrangian
surfaces in this fourfold, see for example \cite{V}.

{\bf Acknowledgements} The authors would like to thank Alexander Kuznetsov, Daniel Huybrechts, Christoph Sorger,
Manfred Lehn and Yukinobu Toda for useful discussions and remarks.

\section{Twisted cubics and sheaves on a cubic fourfold}

\subsection{Twisted cubics on cubic surfaces and determinantal representations}

Let us recall the structure of the general fibre of the map $a\colon M_3(Y)\to Z'$
in (\ref{eqn_contractions}). We follow \cite{LLSvS} in notation and
terminology and we refer to \cite{EPS, LLSvS} for all details about the geometry
of twisted cubics.

Consider a cubic surface $S=Y\cap \bbP^3$ where $\bbP^3$
is a general linear subspace in $\bbP^5$. There exist several families of generalized
twisted cubics on $S$. Each of the families is isomorphic to $\bbP^2$ and these
are the fibres of the map $a$. The number of families depends on $S$. If the surface
is smooth there are 72 families, corresponding to 72 ways to represent $S$ as a
blow-up of $\bbP^2$ (and to the 72 roots in the lattice $E_6$). Each of the families
is a linear system which gives a map to $\bbP^2$. If $S$
is singular, generalized twisted cubics on it can be of two different types. Curves of
the first type are arithmetically Cohen-Macaulay (aCM), and those
of the second type are non-CM. The detailed description of their geometry on surfaces with different
singularity types can be found in \cite{LLSvS}, \S 2. For our purposes it is enough
to recall that the image in $Z'$ of non-CM curves under the map $a$ is exactly the
exceptional divisor of the blow-up $\sigma\colon Z'\to Z$ in (\ref{eqn_contractions}),
see \cite{LLSvS}, Proposition 4.1.

In this section we deal only with aCM curves and we also assume that
the surface $S$ has only ADE singularities. In this case every aCM curve belongs
to a two-dimensional linear system with smooth general member, just as in the case of smooth $S$ \cite[Theorem 2.1]{LLSvS}.
Moreover, these linear systems are in one-to-one correspondence with the determinantal
representations of $S$. Let us explain this in detail.

Let $S$ be a cubic surface in $\bbP^3$ with at most ADE singularities. Let $\alpha \colon S\hrarr \bbP^3$
denote the embedding and let $p\colon \tilde{S}\to S$ be the minimal resolution of singularities.
Take a general aCM twisted cubic $C$ on $S$ and let $\tilde{C} \subset \tilde{S}$ be its proper preimage.
Let $\tilde{L}=\OO_{\tilde{S}}(\tilde{C})$ be the corresponding line bundle and let $L=p_*\tilde{L}$ be its direct image.

\begin{lemma}\label{lem_L} The sheaf $L$ has the following properties:

(1) $H^0(S,L)=\bbC^3$, $H^k(S,L)=0$ for $k \ge 1$; $H^{k}(S, L(-1)) = H^k(S, L(-2)) = 0$ for $k\ge 0$;

(2) We have the following resolution:
\begin{equation}\label{eqn_determinantal}
0\lrarr\OO_{\bbP^3}(-1)^{\oplus 3}\stackrel{A}{\lrarr}\OO_{\bbP^3}^{\oplus 3}\lrarr \alpha_*L\lrarr 0,
\end{equation}
where $A$ is given by a $3\times 3$ matrix of linear forms on $\bbP^3$, and the surface
$S$ is the vanishing locus of $\det A$;

(3) $\IExt^k(L,L) = 0$ for $k\ge 1$.
\end{lemma}
\begin{proof}
We note that the map $\alpha\circ p\colon \tilde{S}\to \bbP^3$ is given by the anticanonical
linear system on $\tilde{S}$, so we will use the notation $K_{\tilde{S}} = \OO_{\tilde{S}}(-1)$.

{\it (1)} First we observe that $\Rder^mp_*\tilde{L} = 0$ for $m\ge 1$. This follows from the long
exact sequence of higher direct images for the triple
\begin{equation}\label{eqn_triple}
0\lrarr\OO_{\tilde{S}}\lrarr\tilde{L}\lrarr \OO_{\tilde{C}}\otimes \tilde{L}\lrarr 0,
\end{equation}
because the singularities of $S$ are rational, so that $\Rder^m p_*\OO_{\tilde{S}} = 0$ for $m \ge 1$ and the map $p$ induces an embedding of
$\tilde{C}$ into $S$, so that $\Rder^m p_*$ vanishes on sheaves supported on $\tilde{C}$ for $m \ge 1$.

Analogously, $\Rder^mp_*\tilde{L}(-1) = \Rder^mp_*\tilde{L}(-2) = 0$ for $m\ge 1$.
Hence it is enough to verify the cohomology vanishing for $\tilde{L}$.

The linear system $|\tilde{L}|$ is two-dimensional and base point free (we refer
to \S 2 of \cite{LLSvS}, in particular Proposition 2.5). We also know the intersection products
$\tilde{L}\cdot \tilde{L} = 1$, $\tilde{L}\cdot K_{\tilde{S}} = -3$ and $K_{\tilde{S}}\cdot K_{\tilde{S}} = 3$. Using Riemann-Roch
we find $\chi(\tilde{L}) = 3$ and $\chi(\tilde{L}(-1)) = \chi(\tilde{L}(-2)) = 0$.
We have $H^0(\tilde{S},\tilde{L}(-1)) = H^0(\tilde{S},\tilde{L}(-2)) = 0$
which is clear from (\ref{eqn_triple}) since $\tilde{L}|_{\tilde{C}} = \OO_{\bbP^1}(1)$ and
$\OO_{\tilde{S}}(1)|_{\tilde{C}}=\OO_{\bbP^1}(3)$. By Serre duality we have $H^2(\tilde{S},\tilde{L}) =
H^0(\tilde{S},\tilde{L}^\vee(-1))^* = 0$, $H^2(\tilde{S},\tilde{L}(-1)) = H^0(\tilde{S},\tilde{L}^\vee)^* = 0$
because $\tilde{L}^\vee$ is the ideal sheaf of $\tilde{C}$,
and $H^2(\tilde{S},\tilde{L}(-2)) = H^0(\tilde{S},\tilde{L}^\vee(1))^* = 0$. The last vanishing
follows from the fact that $C$ is not contained in any hyperplane in $\bbP^3$.
It follows that $H^1(\tilde{S},\tilde{L}) = H^1(\tilde{S},\tilde{L}(-1)) = H^1(\tilde{S},\tilde{L}(-2)) = 0$.

{\it (2)} We decompose the sheaf $\alpha_*L$ with respect to the full exceptional collection
$\Db(\bbP^3) = \langle \OO_{\bbP^3}(-1),\OO_{\bbP^3},\\ \OO_{\bbP^3}(1), \OO_{\bbP^3}(2)\rangle$.
From part {\it (1)} it follows that $\alpha_*L$ is right-orthogonal to $\OO_{\bbP^3}(2)$ and $\OO_{\bbP^3}(1)$.
The left mutation of $\alpha_*L$ through $\OO_{\bbP^3}$ is given by a cone of the morphism
$\OO_{\bbP^3}^{\oplus 3}\to \alpha_*L$ induced by the global sections of $L$.
This cone is contained in the subcategory generated by the exceptional object $\OO_{\bbP^3}(-1)$.
Hence it must be equal to $\OO_{\bbP^3}(-1)^{\oplus 3}[1]$,
and we obtain the resolution (\ref{eqn_determinantal}) for $\alpha_*L$.

{\it (3)} Since $L$ is a vector bundle outside of the singular points of $S$,
the sheaves $\IExt^k(L,L)$ for $k\ge 1$ must have zero-dimensional support.
It follows that it will be sufficient to prove that $\Ext^k(L,L)=0$ for $k\ge 0$.

We first compute $\Ext^k(\alpha_*L,\alpha_*L)$. Applying $\Hom(-,\alpha_*L)$ to (\ref{eqn_determinantal})
we get the exact sequence
$$
0\lrarr\Hom(\alpha_*L,\alpha_*L)\lrarr H^0(\bbP^3,\alpha_*L)^{\oplus 3}\lrarr H^0(\bbP^3,\alpha_*L(1))^{\oplus 3}\lrarr \Ext^1(\alpha_*L,\alpha_*L)\lrarr 0,
$$
where we use that $H^k(\bbP^3,\alpha_*L(m)) = 0$ for $k\ge 1$, $m\ge 0$ which is clear from (\ref{eqn_determinantal}).
This also shows that $\Ext^k(\alpha_*L,\alpha_*L) = 0$ for $k\ge 2$. We have $\dim\Hom(\alpha_*L,\alpha_*L)=1$ and
from the sequence above and (\ref{eqn_determinantal}) we compute $\dim\Ext^1(\alpha_*L,\alpha_*L)=19$.

The object $\Lder \alpha^*\alpha_*L$ is included into the triangle $\Lder \alpha^*\alpha_*L\to L\to L(-3)[2]\to \Lder \alpha^*\alpha_*L[1]$,
see \cite{KM}, Lemma 1.3.1. Applying $\Hom(-,L)$ to this triangle and using 
$\Ext^k(\Lder \alpha^*\alpha_*L,L) = \Ext^k(\alpha_*L,\alpha_*L)$ we get the exact sequence
$$
0\lrarr \Ext^1(L,L)\lrarr \Ext^1(\alpha_*L,\alpha_*L)\lrarr \Hom(L,L(3))\lrarr \Ext^2(L,L)\lrarr 0.
$$
The arrow in the middle is an isomorphism. To see this note that
$\Hom(L,L(3))=H^0(S,N_{S/\bbP^3})=\bbC^{19}$ and that all the deformations of $\alpha_*L$
are induced by the deformations of its support $S$. It follows that $\Ext^1(L,L)=\Ext^2(L,L)=0$.
As we have mentioned above the sheaves $\IExt^k(L,L)$ have zero-dimensional support for $k\ge 1$,
and from the local-to-global spectral sequence we see that $\Ext^k(L,L)=H^0(S,\IExt^k(L,L))$ for $k\ge 1$.
It follows that $\IExt^1(L,L)=\IExt^2(L,L)=0$. To prove the vanishing of higher $\IExt$'s we
construct a quasi-periodic free resolution for $L$. From (\ref{eqn_determinantal}) we see
that the restriction of the complex $\OO_{\bbP^3}(-1)^{\oplus 3}\stackrel{A}{\lrarr}\OO_{\bbP^3}^{\oplus 3}$
to $S$ will have cohomology $L$ in degree $0$ and $L(-3)$ in degree $-1$. Hence $L$ is
quasi-isomorphic to the complex of the form
$$
\ldots\lrarr\OO_S(-7)^{\oplus 3}\lrarr\OO_S(-6)^{\oplus 3}\lrarr\OO_S(-4)^{\oplus 3}\lrarr\OO_S(-3)^{\oplus 3}\lrarr\OO_S(-1)^{\oplus 3}\lrarr\OO_S^{\oplus 3}\lrarr 0.
$$
This complex is quasi-periodic of period two, with subsequent entries obtained by tensoring by $\OO_S(-3)$.
Applying $\IHom(-,L)$ to this complex we see that $\IExt^k(L,L)$ are also quasi-periodic, and
vanishing of the first two of these sheaves implies vanishing of the rest.
\end{proof}

Starting from $L$, we have constructed the determinantal representation of $S$.
Conversely, given a sequence (\ref{eqn_determinantal}),
generalized twisted cubics corresponding to this determinantal representation can
be recovered as vanishing loci of sections of $L$.
More detailed discussion of determinantal representations of cubic surfaces
with different singularity types can be found in \cite{LLSvS}, \S 3.

\subsection{Moduli spaces of sheaves on a cubic fourfold}
Let $S=Y\cap \bbP^3$ be a linear
section of $Y$ with ADE singularities and $L$ a sheaf which gives a determinantal
representation of  $S$ as in (\ref{eqn_determinantal}). Denote by $i\colon S\hrarr Y$
the embedding. We consider the moduli space of torsion sheaves on $Y$ of the form $i_*L$
to get a description of an open subset of $Z$.

\begin{lemma}\label{lem-unobs}
For any $u \in \Ext^1(i_*L, i_*L)$ its Yoneda square $u \circ u \in \Ext^2(i_*L, i_*L)$ is zero,
so that the deformations of $i_*L$ are unobstructed.
\end{lemma}
\begin{proof}
Recall that $L$ is a rank one sheaf on $S$. 
The unobstructedness is clear when
$S$ is smooth, because $L$ is a line bundle in this case. Then the local $\IExt$'s are
given by $\IExt^k(i_*L,i_*L) = i_*\Lambda^kN_{S/Y}$ (see \cite{KM}, Lemma 1.3.2 for the proof of this).
In the case when $S$ is singular and $L$ is not locally free we can use the same argument
as in Lemma 1.3.2 of \cite{KM} to obtain a spectral sequence 
$E_2^{p,q}=i_*(\IExt^p(L,L)\otimes \Lambda^qN_{S/Y}) \Rightarrow \IExt^{p+q}(i_*L,i_*L)$.
Now we can use the second part of Lemma \ref{lem_L} to conclude that in this case
$\IExt^k(i_*L,i_*L) = i_*\Lambda^kN_{S/Y}$ as well.

We have $N_{S/Y}=\OO_S(1)^{\oplus 2}$ and $H^m(S,\OO_S(k)) = 0$ for $k\ge 0$,
$m\ge 1$ and from the local-to-global spectral sequence we deduce that $\Ext^k(i_*L,i_*L)=H^0(S,\Lambda^kN_{S/Y})$.  
The algebra structure is induced
by exterior product $\Lambda^kN_{S/Y}\otimes \Lambda^mN_{S/Y}\to \Lambda^{k+m}N_{S/Y}$ (see \cite{KM}, Lemma 1.3.3). 
The exterior square of any section of 
$N_{S/Y}$ is zero and unobstructedness follows.
\end{proof}

The sheaf $i_*L$ has Hilbert polynomial $P(i_*L,n)=\frac{3}{2} n^2+\frac{9}{2} n +3$ which
is easy to compute from (\ref{eqn_determinantal}).
Denote by $\MM_L$ the irreducible component of the moduli space of semistable sheaves with this
Hilbert polynomial containing $i_*L$.

Let us denote by $V$ the 6-dimensional vector space, so that $Y\subset \bbP(V)=\bbP^5$.
Denote by $G$ the Grassmannian $\Gr(4,V)$.
Recall from \cite{LLSvS} that we have a closed embedding $\mu\colon Y\hrarr Z$,
and the open subset $Z\backslash \mu(Y)$ corresponds to aCM twisted cubics.
There exists a map $\pi: Z\backslash \mu(Y)\to G$ which sends a twisted cubic
to its linear span in $\bbP^5$.
If we consider linear sections $S = Y\cap \bbP^3$, then
$S$ can have non-ADE singularities, but the codimension in $G$ of such linear subspaces
is at least 4 by Proposition 4.2 and Proposition 4.3 in \cite{LLSvS}.
Denote by $G^\circ\subset G$ the open subset consisting of $U\in G$, such that $Y\cap \bbP(U)$ has
only ADE singularities.
Let $Z^\circ = \pi^{-1}(G^\circ)$ be the corresponding open subset in $Z\backslash \mu(Y)$.
This open subset has complement of codimension 4.

\begin{lemma}\label{lem_MML}
There exists an open subset $\MM_L^\circ\hrarr \MM_L$ isomorphic to $Z^\circ$.
The sheaves on $Y$ corresponding to points of $\MM_L^\circ$ are of the form $i_*L$,
where $L$ gives a determinantal representation for a linear section $S=Y\cap\bbP^3$ with ADE
singularities.
\end{lemma}
\begin{proof}

Denote by $\UU$ the universal subbundle of $\OO_G\otimes V$.
Let $p: \bbP(\UU)\to G$ be the projection and $\HH=\mathcal{H}om_{p}(\OO_{\bbP(\UU)}(-1)^{\oplus 3},\OO_{\bbP(\UU)}^{\oplus 3})$.
We have $\HH\simeq(\UU^\vee)^{\oplus 9}$. We will denote by the same letter $\HH$ the total space
of the bundle $\HH$. By construction, over $\HH\times_G\bbP(\UU)$ we have
the universal morphism
$$
\OO_{\bbP(\UU)}(-1)^{\oplus 3}\stackrel{\AA}{\lrarr}\OO_{\bbP(\UU)}^{\oplus 3}.
$$
Denote by $\HH^\circ$ the open subset in the total space of $\HH$ where $\mathrm{det}(\AA)\neq 0$.
Consider the closed embedding $j: \HH^\circ\times_G\bbP(\UU) \hrarr \HH^\circ\times \bbP(V)$
and the sheaf $\MM = \coker(j_*\AA)$ on $\HH^\circ\times \bbP(V)$.
Let $q: \HH^\circ\times \bbP(V)\to \HH^\circ$ be the projection.
For a point $A\in \HH^\circ$ the restriction $\MM|_{q^{-1}(A)}$ is a sheaf that defines a determinantal representation
of a cubic surface in $\bbP(U)\subset \bbP(V)$. The condition that this surface is contained in $Y$
defines a closed subvariety $\WW\subset \HH^\circ$.

Let $\beta: \WW\times Y\hrarr \HH^\circ\times \bbP(V)$ be the closed embedding. Define $\LL=\MM|_{\WW\times Y}$
and consider the open subset $G^\circ\subset G$ of such subspaces $U\subset V$ that $\bbP(U)\cap Y$ has ADE singularities.
Let $\WW^\circ$ be the preimage of $G^\circ$ under the natural map $\WW\to G$.
The sheaf $\LL$ on $\WW^\circ\times Y$ is flat over $\WW^\circ$ since Hilbert polynomials
of its restrictions to the fibres are the same (see \cite{H}, chapter III, Theorem 9.9). We obtain a morphism
$\psi: \WW^\circ\to \MM_L$. Denote its image by $\MM_L^\circ$.
Consider the fibre $\WW_U$ of the map $\WW^\circ\to G^\circ$ over a point $U\in G$ and the restriction
of $\LL$ to $\WW_U\times Y$. Over a point $w\in \WW_U$ the sheaf $\LL$ defines a determinantal representation
of the surface $Y\cap \bbP(U)$. The general structure of determinantal representations (see \cite{LLSvS} \S 3)
implies that each connected component of the fibre $\WW_U$ is a single $(\mathrm{GL}_3\times \mathrm{GL}_3)/\bbC^*$ orbit
(\cite{LLSvS} Corollary 3.7).
Connected components of $\WW_U$ are in one-to-one correspondence with non-isomorphic determinantal representations
of $Y\cap \bbP(U)$. The restriction of $\LL$ to each connected component of $\WW_U\times Y$ is a constant family of sheaves,
so the map $\psi$ contracts connected components of the fibre $\WW_U$. From the explicit description of $Z^\circ$ given
above, we see that $\MM_L^\circ$ is isomorphic to $Z^\circ$.
The properties stated in the lemma
are clear from construction. We also see that $\WW^\circ$ is a $(\mathrm{GL}_3\times \mathrm{GL}_3)/\bbC^*$-fibre bundle
over $Z^\circ$.
\end{proof}

The sheaves $i_*L$ are not contained in the subcategory $\AA_Y$. 
In order to show that the
closed 2-form described in \cite{KM} is a symplectic form on $\MM_L^\circ$, we are going to project the
sheaves $i_*L$ to $\AA_Y$, and then show that this projection
induces an isomorphism of open subsets of moduli spaces respecting the 2-forms (up to a sign).

\begin{lemma}\label{lem_proj}
The sheaves $i_*L$ are globally generated and lie in the
subcategory $\langle\AA_Y,\OO_Y\rangle$. 
The space of global sections $H^0(Y, i_*L)$ is three-dimensional, and the sheaf $F_L$, defined by the exact triple
\begin{equation}\label{eqn_F}
0\lrarr F_L\lrarr \OO_Y^{\oplus 3}\lrarr i_*L\lrarr 0.
\end{equation}
lies in $\AA_Y$.
\end{lemma}
\begin{proof}
From Lemma \ref{lem_L} we deduce that $i_*L$ is right orthogonal to $\OO_Y(1)$, $\OO_Y(2)$, so that $i_*L$ lies in $\langle\AA_Y,\OO_Y\rangle$.
It also follows from Lemma \ref{lem_L} that $i_*L$ is globally generated, the global sections are three-dimensional and that
the higher cohomology groups of $L$ vanish. Thus $F_L$ is (up to a shift)
the left mutation of $i_*L$ through the exceptional bundle $\OO_Y$, and in particular it lies in $\AA_Y$.
\end{proof}

\begin{lemma}\label{lem_F}
Consider the exact triple (\ref{eqn_F}) where $i_*L$ is in $\MM_L^\circ$. Then
$F_L$ is a Gieseker-stable rank 3 sheaf contained in $\AA_Y$ with Hilbert polynomial $P(F_L, n) = \frac38 n^4 + \frac94 n^3 + \frac{33}{8} n^2 + \frac94 n$.
\end{lemma}
\begin{proof}
By Lemma \ref{lem_MML} the sheaf $i_*L$ is right-orthogonal to $\OO_Y(2)$ and $\OO_Y(1)$.
The sheaf $F_L$ is a shift of the left mutation of $i_*L$ through $\OO_Y$, hence it is
contained if $\AA_Y$. The Hilbert polynomial can be computed using the Hirzebruch-Riemann-Roch formula.
It remains to check the stability of $F_L$.

The sheaf $F_L$ is a subsheaf of $\OO_Y^{\oplus 3}$, hence it has no torsion. In order
to check the stability we consider all proper saturated subsheaves $\GG\subset F_L$.
We have to make sure that $p(\GG, n)< p(F_L, n)$ where $p$ is the reduced Hilbert polynomial
(see \cite{HL} for all the relevant definitions). We use the convention
that the inequalities between polynomials are supposed to hold for $n \gg 0$.

We denote by $P$ the non-reduced Hilbert polynomial. We have $P(\OO_Y, n) = a_0 n^4 + a_1 n^3 +\ldots + a_4$,
with the leading coefficient $a_0 = \frac{3}{4!}$. From the exact sequence
(\ref{eqn_F}) we see that $P(F_L, n) = 3P(\OO_Y,n) - P(i_*L,n)$.
Since $i_*L$ has two-dimensional support, the degree of $P(i_*L,n)$ is two, and hence
the leading coefficient of $P(F_L, n)$ equals $3a_0$. So we have
\begin{equation}\label{pFF_C}
p(F_L, n) = p(\OO_Y, n) - \frac{1}{3a_0}P(i_*L, n).
\end{equation}

Let $\tilde{\GG}$ be the saturation of $\GG$ inside $\OO_Y^{\oplus 3}$.
Then $\tilde{\GG}$ is a reflexive sheaf and we have a diagram:
$$
\begin{tikzcd}[]
0\rar & \GG\dar\rar   & \tilde{\GG}\dar\rar   & \HH\dar\rar  & 0\\
0\rar & F_L\rar       & \OO_Y^{\oplus 3}\rar     & i_*L\rar & 0
\end{tikzcd}
$$
In this diagram $\HH$ is a torsion sheaf which injects into $i_*L$ because $F_L/\GG$ is torsion-free.
Note that $\OO_Y^{\oplus 3}$ is Mumford-polystable, so $c_1(\GG)\le c_1(\tilde{\GG})\le 0$. If $c_1(\GG)< 0$
then $\GG$ is not destabilizing in $F_L$ because $c_1(F_L) = 0$.

Next we consider the case $c_1(\GG)= c_1(\tilde{\GG})= 0$. In this case $\tilde{\GG}=\OO_Y^{\oplus m}$ where $m=1$ or $m=2$.
This is clear if $\rk{\tilde{\GG}}=1$ since a reflexive sheaf of rank one is a line bundle.
If $\rk{\tilde{\GG}}=2$ we can consider the quotient $\OO_Y^{\oplus 3}/\tilde{\GG}$ which is torsion-free,
globally generated, of rank one and has zero first Chern class. It follows that the quotient is 
isomorphic to $\OO_Y$ and then $\tilde{\GG}=\OO_Y^{\oplus 2}$.

We have an exact triple $0\lrarr\GG\lrarr\OO_Y^{\oplus m}\lrarr\HH\lrarr 0$ with $m$ equal to $1$ or $2$.
We see that $p(\GG,n) = p(\OO_Y,n) - \frac{1}{ma_0}P(\HH, n)$. Note that $\HH$ is a non-zero sheaf
which injects into $i_*L$, and the sheaf $L$ on the surface $S$ is torsion-free of rank one. Hence the
leading coefficient of $P(\HH, n)$ is the same as for $P(i_*L, n)$ and this implies
$\frac{1}{ma_0}P(\HH, n)> \frac{1}{3a_0}P(i_*L, n)$. From this and (\ref{pFF_C}) we conclude that
$p(\GG,n)<p(F_L,n)$, hence $\GG$ is not destabilizing. This completes the proof.
\end{proof}

Let us consider the moduli space of rank 3 semistable sheaves on $Y$ with Hilbert polynomial $P(F_L,n)$.
Denote by $\MM_F$ its irreducible component which contains the sheaves $F_L$ from (\ref{eqn_F}).

\begin{lemma}\label{lem_mutation}
The left mutation of $i_*L$ through $\OO_Y$ gives an open embedding $\MM_L^\circ\to \MM_F$.
\end{lemma}
\begin{proof}
Recall from the proof of Lemma \ref{lem_MML} that $\MM_L^\circ$ was defined as the image of a map $\WW^\circ\to \MM_L$
where $\WW^\circ$ was a fibre bundle over $Z^\circ$. On $X = \WW^\circ\times Y$ a universal
sheaf $\LL$ flat over $\WW^\circ$ was constructed. Denote by $\pi\colon X\to \WW^\circ$ the projection.

By definition of $\MM_L^\circ$ and from Lemma \ref{lem_L} it follows that $\pi_*\LL$
is a rank 3 vector bundle and we have an exact sequence $0\to\FF_\LL\to\pi^*\pi_*\LL\to \LL\to 0$.
The family of sheaves $\FF_\LL$ defines a map $\WW^\circ\to \MM_F$ which factors
through $\MM_L^\circ\to \MM_F$. We will show
that the differential of the latter map is an isomorphism.

For a sheaf $i_*L$ corresponding to a point of $\MM_L^\circ$ and any tangent vector
$u\in\Ext^1(i_*L,i_*L)$ we have unique morphism of triangles
\begin{equation}\label{eqn_mutation}
\begin{tikzcd}[]
F_L \dar{u'}\rar & \OO_Y^{\oplus 3} \dar{0}\rar & i_*L \dar{u}\rar & F_L[1] \dar{u'[1]} \\
F_L[1] \rar      & \OO_Y^{\oplus 3}[1] \rar     & i_*L[1]\rar      & F_L[2]
\end{tikzcd}
\end{equation}
Uniqueness of $u'$ follows from $\Ext^1(\OO_Y,F_L) = 0$. Moreover, $u$ is uniquely
determined by $u'$ because $\Ext^1(i_*L,\OO_Y) = \Ext^3(\OO_Y,i_*L(-3))^* = 0$. This
shows that the mutation induces an isomorphism of $\Ext^1(i_*L,i_*L)$ and $\Ext^1(F_L,F_L)$.

Finally, let us prove that the map $\MM_L^\circ\to \MM_F$ is injective. It follows from 
Grothendieck-Verdier duality that $\IExt^2(i_*L,\OO_Y)=i_*L^\vee(2)$.
Then from (\ref{eqn_F}) we see that $\IExt^1(F_L,\OO_Y)=i_*L^\vee(2)$ and hence 
$L$ can be reconstructed from $F_L$.
\end{proof}

\subsection{The symplectic form and Lagrangian subvarieties}

Let us recall the description of the two-form on the moduli spaces of sheaves on $Y$ from \cite{KM}.

Given a coherent sheaf $\FF$ on $Y$ we can define its Atiyah class $\Ati{\FF}\in\Ext^1(\FF,\FF\otimes\Omega_Y)$.
The Atiyah class is functorial, meaning that for any morphism of sheaves $\alpha\colon \FF\to \GG$
we have $\Ati{\GG}\circ\alpha=(\alpha\otimes \mathrm{id})\circ\Ati{\FF}$.

We define a bilinear form $\sigma$ on the vector space $\Ext^1(\FF,\FF)$. Given two elements $u,v\in \Ext^1(\FF,\FF)$
we consider the composition $\Ati{\FF}\circ u\circ v\in \Ext^3(\FF,\FF\otimes\Omega_Y)$ and apply the
trace map $\Tr\colon\Ext^3(\FF,\FF\otimes\Omega_Y)\to \Ext^3(\OO_Y,\Omega_Y)= H^{1,3}(Y)= \bbC$ to it:
\begin{equation}\label{eqn_sigma}
\sigma(u,v)=\Tr(\Ati{\FF}\circ u\circ v).
\end{equation}

Note that when the Kuranishi space of $\FF$ is smooth then for any $u\in \Ext^1(\FF,\FF)$ we have
$u\circ u=0$ and then $\sigma(u,u)=0$. In this case $\sigma$ is antisymmetric. Hence the formula
(\ref{eqn_sigma}) defines a two-form at smooth points of moduli spaces of sheaves on $Y$. This form
is closed by \cite{KM}, Theorem 2.2. 

\begin{lemma}\label{lem_symplectic}
The formula (\ref{eqn_sigma}) defines a symplectic form on $\MM_L^\circ$ which
coincides up to a non-zero constant with the restriction of the symplectic form on $Z$ under the isomorphism $\MM_L^\circ\simeq Z^\circ$.
\end{lemma}
\begin{proof}
By Lemma \ref{lem-unobs} the sheaves $i_*L$ from $\MM_L^\circ$ have unobstructed deformations,
so that (\ref{eqn_sigma}) indeed defines a two-form. 

Recall from Lemma \ref{lem_mutation} that we have an open embedding $\MM_L^\circ\hookrightarrow \MM_F$.
Let us show that this embedding respects (up to a sign) symplectic forms on $\MM_L$ and $\MM_F$ given by (\ref{eqn_sigma}).
Note that by functoriality of Atiyah classes the following diagram gives a morphism
of triangles:
$$
\begin{tikzcd}[]
F_L \dar{\Ati{F_L}}\rar & \OO_Y^{\oplus 3} \dar{\Ati{\OO_Y^{\oplus 3}} = 0}\rar &  i_*L \dar{\Ati{i_*L}}\rar & F_L[1] \dar{\Ati{F_L}[1]} \\
F_L\otimes\Omega_Y[1] \rar & \Omega_Y^{\oplus 3}[1] \rar & i_*L\otimes\Omega_Y[1]\rar & F_L\otimes\Omega_Y[2]
\end{tikzcd}
$$
For any pair of tangent vectors $u,v\in \Ext^1(i_*L,i_*L)$ we have two morphisms of triangles as in (\ref{eqn_mutation}).
If we compose these two morphisms of triangles with the one induced by Atiyah classes then we get the following:
$$
\begin{tikzcd}[]
F_L \dar{\Ati{F_L}\circ u'\circ v'}\rar & \OO_Y^{\oplus 3} \dar{0}\rar &
 i_*L \dar{\Ati{i_*L}\circ u\circ v}\rar & F_L[1] \dar{\Ati{F_L}\circ u'\circ v'[1]} \\
F_L\otimes\Omega_Y[3] \rar & \Omega_Y^{\oplus 3}[3] \rar & i_*L\otimes\Omega_Y[3]\rar & F_L\otimes\Omega_Y[4]
\end{tikzcd}
$$
This diagram is a morphism of triangles and the additivity of traces implies that $\sigma(u,v)=-\sigma(u',v')$.

By Theorem 4.3 from \cite{KM} the form $\sigma$ on $\MM_F$ is symplectic, because
the sheaves $F_L$ are contained in $\AA_Y$. Hence $\sigma$ is a symplectic form
on $\MM_L^\circ$. But $\MM_L^\circ$ is embedded into $Z$ as an open subset with
complement of codimension four. This implies that the symplectic form on
$\MM_L^\circ$ is unique up to a constant, because $Z$ is IHS. This completes the
proof.
\end{proof}

\begin{thm}\label{thm_MMF}
The component $\MM_F$ of the moduli space of Gieseker stable sheaves with Hilbert polynomial $P(F_L, n)$ 
is birational to the IHS manifold $Z$. Under this birational equivalence 
the symplectic form on $Z$ defined in \cite{LLSvS} corresponds
to the Kuznetsov-Markushevich form on $\MM_F$.
\end{thm}
\begin{proof}
Follows from Lemmas \ref{lem_MML}, \ref{lem_F}, \ref{lem_mutation}, \ref{lem_symplectic}.
\end{proof}

Now we explain how hyperplane sections of $Y$ give rise to Lagrangian
subvarieties of $Z$.
Let $H\subset \bbP^5$ be a generic hyperplane, so that $Y_H=Y\cap H$ is a smooth
cubic threefold. Twisted cubics contained in $Z$ form a subvariety $M_3(Y)_H\subset M_3(Y)$
whose image in $Z$ we denote by $Z_H$. Its open subset $Z_H^\circ = Z_H\cap Z^\circ$
consists of sheaves $i_*L$ whose support is contained in $H$.

\begin{prop}\label{prop_Lagr}
$Z_H$ is a Lagrangian subvariety of $Z$.
\end{prop}
\begin{proof}
It is clear that $Z_H$ has dimension four since the Grassmannian of three-dimensional
subspaces in $H$ is $\bbP^4$.
Consider a sheaf $i_*L$ whose support $S$ is smooth and contained in $Y_H$.
Since $L$ is a locally free sheaf on $S$ we have $\IExt^k(i_*L,i_*L)=i_*\Lambda^kN_{S/Y}$
(see for example \cite{KM}, Lemma 1.3.2). The higher cohomologies of the sheaves $\IExt^k(i_*L,i_*L)$
vanish for $k\ge 0$, because $N_{S/Y}=\OO_S(1)^{\oplus 2}$ and the sheaves $\OO_S(k)$
have no higher cohomologies for $k\ge 0$. Hence from the local-to-global spectral
sequence we find that $T_{i_*L}\MM_L=\Ext^1(i_*L, i_*L)=H^0(S,N_{S/Y})$. Moreover, the Yoneda
multiplication on $\Ext$'s is given by the map $H^0(S,N_{S/Y})\times H^0(S,N_{S/Y})\to H^0(S,\Lambda^2N_{S/Y})$
which is induced from the exterior product morphism $N_{S/Y}\otimes N_{S/Y}\to \Lambda^2N_{S/Y}$
(see \cite{KM}, Lemma 1.3.3).
Now, the tangent space to $Z_H$ at $i_*L$ is $H^0(S,N_{S/Y_H})$. But the exterior product
$N_{S/Y_H}\otimes N_{S/Y_H}\to \Lambda^2N_{S/Y_H}=0$ vanishes because $N_{S/Y_H}$ is of rank one.
So the Yoneda product vanishes on the corresponding subspace of $\Ext^1(i_*L, i_*L)$
and from the definition of the symplectic form (\ref{eqn_sigma}) we conclude that the
tangent subspace to $Z_H$ is Lagrangian. This holds on an open subset of
$Z_H$, so $Z_H$ is a Lagrangian subvariety. 
\end{proof}

In the next section we give a description of the subvarieties $Z_H$ in
terms of intermediate Jacobians of the threefolds $Y_H$.

\section{Twisted cubics on a cubic threefold}


In this section we assume that the cubic fourfold $Y$ and its hyperplane section $Y_H$ are
chosen generically, so that $Y_H$ is smooth and all the surfaces obtained by intersecting $Y_H$ with
three-dimensional subspaces have at worst ADE singularities. For general $Y$ and $H$ this indeed
will be the case, because for a general cubic threefold in $\bbP^4$ its hyperplane sections
have only ADE singularities. One can see this from dimension count by considering the codimensions
of loci of cubic surfaces with different singularity types (see for example \cite{LLSvS}, sections 2.2 and 2.3).

The cubic threefold $Y_H$ has an intermediate Jacobian $\IJac{Y_H}$ which is a principally polarized
abelian variety. 
We will show that if we choose a general hyperplane $H$ then the Abel-Jacobi map 
\[
\AJ\colon Z_H\to \IJac{Y_H}
\]
defines
a closed embedding on an open subset $Z_H^\circ$ and the complement $Z_H\backslash Z_H^\circ$ is contracted
to a point. 
The image of $\AJ$ is the theta-divisor $\Theta \subset \IJac{Y_H}$.

Recall from the description of $Z$ that we have an embedding $\mu\colon Y\hookrightarrow Z$.
We have $Z_H^\circ \simeq Z_H\setminus \mu(Y)$ and $Z_H\cap \mu(Y)\simeq Y_H$.
Hence the Abel-Jacobi map $\AJ\colon Z_H\to \IJac{Y_H}$
gives a resolution of the unique singular point of the theta-divisor and the exceptional
divisor of this map is isomorphic to $Y_H$. This explicit description of the
singularity of the theta-divisor first obtained in \cite{B2} implies 
Torelli theorem for cubic threefolds.

The fact that $Z_H$ is birational to the theta-divisor in $\IJac{Y_H}$ also
follows from \cite{I} (see also \cite[Proposition 4.2]{B1}).

\subsection{Differential of the Abel-Jacobi map}

As before, we will identify the open subset $Z_H^\circ$ with an open subset in the moduli space
of sheaves of the form $i_*L$, where $i\colon S\hookrightarrow Y_H$ is a hyperplane
section and $L$ is a sheaf which gives a determinantal representation (\ref{eqn_determinantal}) of this section.

The Abel-Jacobi map $\AJ\colon Z_H^\circ\to \IJac{Y_H}$ can be described as follows.
We use the Chern classes with values in the Chow ring $\CH(Y_H)$.
The second Chern class $c_2(i_*L) \in \CH^2(Y_H)$ is a cycle class of degree $3$.
Let $h\in \CH^1(Y_H)$ denote the class of a hyperplane section, then
$c_2(i_*L)-h^2$ is a cycle class homologous to zero, and it defines an element in the intermediate Jacobian.

Since $c_2(i_* L)$ can be represented by corresponding twisted cubics, the map above extends to $AJ: Z_H \to \IJac{Y_H}$.

\begin{lemma}\label{lem_dAJ}
The differential of the Abel-Jacobi map $d\AJ_{i_*L}\colon \Ext^1(i_*L,i_*L)\to H^{1,2}(Y_H)$
at the point corresponding to the sheaf $i_*L$ is given by
\begin{equation}\label{eqn_dAJ}
d\AJ_{i_*L}(u) = \frac12 \Tr(\Ati{i_*L}\circ u),
\end{equation}
for any $u\in \Ext^1(i_*L,i_*L)$.
\end{lemma}
\begin{proof}
We apply the general formula for the derivative of the Abel-Jacobi map, see Appendix \ref{appendix}, Proposition \ref{prop_dAJ_general}.
We have $c_1(i_*L) = 0$, so that $s_2(i_*L) = -2c_2(i_*L)$, which yields the $\frac12$ factor in the statement.
\end{proof}

It will be convenient for us to rewrite (\ref{eqn_dAJ}) in terms of the linkage class of a sheaf, see \cite{KM}.
We recall its definition in our particular case of the embedding $j\colon Y_H\hookrightarrow \bbP^4$.
If $\FF$ is a sheaf on $Y_H$ then the object $j^*j_*\FF\in \Db(Y_H)$ has non-zero cohomologies only in degrees
$-1$ and $0$. They are equal to $\FF\otimes N_{Y/\bbP^4}^\vee = \FF(-3)$ and $\FF$ respectively.
Hence the triangle
\[
\FF(-3)[1]\lrarr \Lder j^*j_*\FF\lrarr \FF\lrarr \FF(-3)[2]. 
\]
The last morphism in
this triangle is called the linkage class of $\FF$ and will be denoted by $\epsilon_\FF\colon \FF\to \FF(-3)[2]$.
The linkage class can also be described as follows (see \cite{KM}, Theorem 3.2): 
let us denote by $\kappa\in \Ext^1(\Omega_{Y_H},\OO_{Y_H}(-3))$
the extension class of the conormal sequence $0\to\OO_{Y_H}(-3)\to\Omega_{\bbP^4}|_{Y_H}\to \Omega_{Y_H}\to 0$; then
we have $\epsilon_{\FF} = (\mathrm{id}_{\FF}\otimes \kappa)\circ\Ati{\FF}$.

Note that composition with $\kappa$ gives an isomorphism of vector spaces $H^{1,2}(Y_H)=\Ext^2(\OO_{Y_H},\Omega_{Y_H})$
and $\Ext^3(\OO_{Y_H}, \OO_{Y_H}(-3)) = H^0(Y_H,\OO_{Y_H}(1))^*$. Composing the right hand side
of (\ref{eqn_dAJ}) with $\kappa$ and using the fact that taking traces commutes with compositions,
we obtain the following expression for $d\AJ(u)$ where $u\in\Ext^1(i_*L,i_*L)$:
\begin{equation}\label{eqn_dAJ2}
\kappa \circ d\AJ_{i_*L}(u) = \frac12 \Tr(\epsilon_{i_*L}\circ u) \in H^0(Y_H, \OO_H(1))^*
\end{equation}

\begin{prop}\label{prop_dAJ}
The differential of the Abel-Jacobi map (\ref{eqn_dAJ}) is injective.
\end{prop}
\begin{proof}
As before, we will denote by $i\colon S\hrarr Y_H$ and $j\colon Y_H\hrarr \bbP^4$ the
embeddings. A point of $Z_H^\circ$ is represented by a sheaf $i_*L$. 
Let us also use the notation $\FF = i_*L$.
It suffices to show that the map $u \mapsto \kappa \circ d\AJ_{i_*L}(u)$ is injective.

The proof is done in three steps.

{\it Step 1.} Let us construct a locally free resolution of $j_*\FF$.
We decompose $j_*\FF$ with respect to the exceptional collection
$\OO_{\bbP^4}(-2)$, $\OO_{\bbP^4}(-1)$, $\OO_{\bbP^4}$, $\OO_{\bbP^4}(1)$, $\OO_{\bbP^4}(2)$.
The sheaf $j_*\FF$ is
already left-orthogonal to $\OO_{\bbP^4}(2)$ and $\OO_{\bbP^4}(1)$ (see Lemma \ref{lem_L}).
It is globally generated by (\ref{eqn_determinantal}) and its left
mutation is the shift of the sheaf $\KK$ from the exact triple
$0\lrarr\KK\lrarr\OO_{\bbP^4}^{\oplus 3}\lrarr j_*\FF\lrarr 0$.
From cohomology exact sequence we see that $H^0(\bbP^4,\KK(1)) = \bbC^6$ and $H^k(\bbP^4,\KK(1))=0$
for $k\ge 1$. We can also check that $\KK(1)$ is globally generated (it is in fact
Castelnuovo-Mumford $0$-regular, as one can see using (\ref{eqn_determinantal})).
The left mutation of $\KK$ through $\OO_{\bbP^4}(-1)$ is the cone of
the surjection $\OO_{\bbP^4}(-1)^{\oplus 6}\to \KK$, and it lies in the subcategory generated
by $\OO_{\bbP^4}(-2)$. Since it has rank 3, this completes the construction of the
resolution for $j_*\FF$.
The resulting resolution is:
\begin{equation}\label{eqn_resjF}
0\lrarr\OO_{\bbP^4}(-2)^{\oplus 3}\lrarr \OO_{\bbP^4}(-1)^{\oplus 6}\lrarr \OO_{\bbP^4}^{\oplus 3}\lrarr j_*\FF\lrarr 0.
\end{equation}

{\it Step 2.} Let us show that the linkage class $\epsilon_\FF$ induces an isomorphism
\[
\Ext^1(\FF,\FF) \to \Ext^3(\FF,\FF(-3)).
\]
The object $\Lder j^*j_*\FF$ is included into the triangle
$$
\Lder j^*j_*\FF\lrarr \FF\stackrel{\epsilon_\FF}{\lrarr} \FF(-3)[2]\lrarr \Lder j^*j_*\FF[1].
$$
Applying $\Hom(\FF,-)$ to this triangle we find the following exact sequence:
$$
\Ext^1(\FF,\Lder j^*j_*\FF)\lrarr \Ext^1(\FF,\FF)\stackrel{\epsilon_\FF\circ-}{\lrarr} \Ext^3(\FF,\FF(-3))
\lrarr \Ext^2(\FF,\Lder j^*j_*\FF).
$$
Note that by (\ref{eqn_resjF}) the object $\Lder j^*j_*\FF$ is represented by a complex of the form
$0\to \OO_{Y_H}(-2)^{\oplus 3}\to \OO_{Y_H}(-1)^{\oplus 6}\to \OO_{Y_H}^{\oplus 3}\to 0$.
Let us check that $\Ext^2(\FF,\Lder j^*j_*\FF) = 0$. By Serre duality $\Ext^q(\FF,\OO_{Y_H}(-p))
=\Ext^{3-q}(\OO_{Y_H}(-p),\FF(-2))^* = H^{3-q}(Y_H,\FF(p-2))^*$ and from (\ref{eqn_determinantal})
we see that for $p=0$ and $1$ these cohomology groups vanish, and for $p=2$ the only
non-vanishing group corresponds to $q=3$. The spectral sequence computing 
$\Ext^k(\FF,\Lder j^*j_*\FF)$, obtained from the complex representing $\Lder j^*j_*\FF$, implies
that $\Ext^k(\FF,\Lder j^*j_*\FF)=0$ for $k\neq 1$ and $\Ext^1(\FF,\Lder j^*j_*\FF)=H^0(Y_H,\FF)^*=\bbC^3$.

We conclude that the map
$\Ext^1(\FF,\FF)\stackrel{\epsilon_\FF\circ-}{\lrarr} \Ext^3(\FF,\FF(-3))$ is surjective.
It is actually an isomorphism, because the vector spaces are of the same dimension.
The dimensions can be computed in the same way as in the proof of Lemma \ref{lem_symplectic}.

{\it Step 3.} 
Let us show that $\Tr: \Ext^3(\FF, \FF(-3)) \to H^3(Y_H, \OO_{Y_H}(-3))$ is injective.


Using Serre duality we identify the dual to the trace map with
\[
\Tr^*: H^0(Y_H, \OO_{Y_H}(1)) \to \Hom(\FF, \FF(1)).
\]
One can show as in the proof of Lemma \ref{lem-unobs} that $\Hom(\FF, \FF(1)) = H^0(S, \OO(1))$ and postcomposing
$\Tr^*$ with this isomorphism gives the restriction map
\[
H^0(Y_H,\OO_{Y_H}(1))\to H^0(S,\OO_S(1)) 
\]
which is surjective.

We see that the composition
\[
\Ext^1(\FF, \FF) \to \Ext^3(\FF, \FF(-3)) \to H^3(Y_H, \OO(-3)) 
\]
is injective and the proof is finished by means of formula (\ref{eqn_dAJ2}).
\end{proof}

\subsection{Image of the Abel-Jacobi map} 


\begin{thm} \label{thm_AJ}
Assume that $Y_H$ is smooth and all its hyperplane sections have at worst ADE singularities.
Then the image of the Abel-Jacobi map $\AJ\colon Z_H\to \IJac{Y_H}$ is the theta-divisor $\Theta \subset \IJac{Y_H}$.
The map $\AJ$ is an embedding on $Z_H^\circ$ and contracts the divisor $Y_H = Z_H \backslash Z_H^\circ$ 
to the unique singular point of $\Theta$.
\end{thm}
\begin{proof}
The divisor $Y_H$ is contracted by the Abel-Jacobi map 
to a point because $Y_H$ is a cubic threefold which has no global one-forms.

To identify the image of $\AJ$ it is enough to check that a general point of $Z_H$ is mapped to a point of $\Theta$.
General point $z\in Z_H$ is represented by a smooth twisted cubic $C$ on a smooth hyperplane section $S\subset Y_H$.
Denote by $C_0\subset S$ a hyperplane section of $S$. Then $C-C_0$ is a degree zero cycle on $Y_H$ and $z$ is
mapped to the corresponding element of the intermediate Jacobian. The cohomology class $[C-C_0]\in H^2(S,\bbZ)$ is
orthogonal to the class of the canonical bundle $K_S$ and has square $-2$. Hence it is a root in the $E_6$ lattice.
All such cohomology classes can be represented by differences of pairs of lines $l_1-l_2$ in 6 different
ways. 

Recall that the Fano variety of lines on the
cubic threefold $Y_H$ is a surface which we will denote by $X$. It was shown in \cite{CG} that
the theta divisor $\Theta\subset \IJac{Y_H}$ can be described as the image of the map $X\times X\to \IJac{Y_H}$
which sends a pair of lines $(l_1,l_2)$ to the point in $\IJac{Y_H}$ corresponding to degree zero
cycle $l_1-l_2$. 
The map $X \times X \to \Theta$ has degree $6$.
We get a commutative diagram:
$$
\begin{tikzcd}[]
X \times X \arrow[dashed]{d}[swap]{6:1} \arrow{r}{6:1} & \Theta \\
Z_H \arrow{ur}[swap]{\AJ} &   
\end{tikzcd}
$$

It follows from the diagram above that $\AJ$ is generically of degree one.
Since $\AJ$ is \'etale on $Z_H^\circ$ by Proposition \ref{prop_dAJ} and the theta-divisor $\Theta$ is a normal variety \cite[Proposition 2, \S3]{B2}
we deduce that $\AJ: Z_H^\circ \to \Theta$ is on open embedding. This completes the proof.
\end{proof}

\appendix

\section{Differential of the Abel-Jacobi map}\label{appendix}


Let $X$ be a smooth complex projective variety of dimension $n$. 
Recall that the
$p$-th intermediate Jacobian of $X$ is the complex torus 
$$
\Jac^p(X) = H^{2p-1}(X,\bbC)/(F^pH^{2p-1}(X,\bbC)+H^{2p-1}(X,\bbZ)),
$$
where $F^{\sdot}$ denotes the Hodge filtration. 
We use the Abel-Jacobi map \cite[Appendix A]{G2}
\[
\AJ^p\colon \CH^p(X,\Z)_{h}\to \Jac^p(X)
\]
where $\CH^p(X)_{h}$ is the group of homologically trivial codimension $p$ algebraic cycles on $X$
up to rational equivalence.

For a coherent sheaf $\FF_0$ on $X$ we consider integral characteristic classes
\[
s_p(\FF_0) = p! \cdot ch_p(\FF_0) \in CH^p(X,\Z)
\]
where $ch_p(\FF_0)$ is the $p$'th component of the Chern character $ch(\FF_0)$. 
These classes can be expressed in terms of the Chern classes using Newton's formula \cite[\S16]{MS}.

Let us consider a deformation of $\FF_0$ over a smooth base $B$ with base point $0\in B$,
that is a coherent sheaf $\FF$ on $X\times B$ flat over $B$ and with $\FF_0 \simeq \FF|_{\pi_B^{-1}(0)}$.
We will denote by $\pi_B$ and $\pi_X$ the two projections from $X\times B$
and by $\FF_t$ the restriction of $\FF$ to $\pi_B^{-1}(t)$, $t\in B$. 
In this setting the difference $s_p(\FF_t)-s_p(\FF_0)$ is contained in $\CH^p(X, \Z)_{h}$
and we get an induced Abel-Jacobi map 
\[
\AJ^p_{\FF}: B \to J^p(X). 
\]

Since the classes $s_p$ are additive, it follows that 
if $0 \to \FF' \to \FF \to \FF'' \to 0$ is a short exact sequences of sheaves on $X \times B$ flat over $B$, then
\begin{equation}\label{AJ-add}
\AJ^p_{\FF} = \AJ^p_{\FF'} + \AJ^p_{\FF''}.
\end{equation}

Recall that a coherent sheaf $\FF_0$ has an Atiyah class $\Ati{\FF_0}\in \Ext^1(\FF_0,\FF_0\otimes\Omega_X)$ \cite[1.6]{KM}.
The vector space $\bigoplus_{p,q\ge 0}\Ext^q(\FF_0,\FF_0\otimes\Omega_X^p)$ has the structure of a
bi-graded algebra with multiplication induced by Yoneda product of $\Ext$'s and exterior
product of differential forms and this defines the $p$'th power of the Atiyah class
\[
\Ati{\FF_0}^p \in \Ext^p(\FF_0,\FF_0\otimes\Omega^p_X).
\]

Given any tangent vector $v\in T_{0}B$ we shall denote its Kodaira-Spencer class
by $\mathrm{KS}_{\FF_0}(v)\in \Ext^1(\FF_0,\FF_0)$
and we consider the composition $\Ati{\FF_0}^p \circ \mathrm{KS}_{\FF_0}(v) \in \Ext^{p+1}(\FF_0, \FF_0 \otimes \Omega_X^p)$.

We will also use the trace maps \cite[1.2]{KM}
$$\Tr\colon \Ext^q(\FF_0,\FF_0\otimes\Omega_X^p)\to \Ext^q(\OO_X,\Omega_X^p)=H^{p,q}(X).$$

\begin{prop}\label{prop_dAJ_general}
In the above setting the differential of the Abel-Jacobi map $\AJ_{\FF}^p: B \to \Jac^p(X)$, $p \ge 2$ at $0 \in B$ is given by
\begin{equation}\label{eqn_dAJ_general}
d\AJ^p_{\FF,0}(v)=\Tr\bigl( (-1)^{p-1}\Ati{\FF_0}^{p-1} \circ \mathrm{KS}_{\FF_0}(v)\bigr),
\end{equation}
for any $v\in T_{0}B$. The right hand side is an element of 
$H^{p-1,p}(X) \subset H^{2p-1}(X,\bbC)/F^pH^{2p-1}(X,\bbC)$.
\end{prop}
\begin{proof}
We argue by induction on the length of a locally free resolution of $\FF$. 
The base of induction is the case when $\FF_0$ is a vector bundle. Then the result is essentially contained
in the paper of Griffiths \cite{G1} (in particular formula 6.8). We will show how to do
the induction step. We note that the statement is local, so we may replace the base
$B$ by an open neighborhood of $0\in B$ every time it is necessary. In particular we assume that $B$ is affine.

By our assumptions $X$ is projective and we denote by $\OO_X(1)$ an ample line bundle.
Then we can find $k$ big enough, so that $\FF(k)$ is generated by global sections
and has no higher cohomology. We define
a sheaf $\GG$ on $X\times B$ as the kernel of the natural map:
$$
0\lrarr\GG\lrarr\pi_B^*{\pi_B}_*(\FF(k))\otimes\OO_X(-k)\lrarr \FF\lrarr 0.
$$
Since $\FF$ is flat over $B$ and ${\pi_B}_*(\FF_0(k))$ 
is a vector bundle on $B$ for $k$ large enough \cite[Proof of Theorem 9.9]{H},
the sheaf $\GG$ is flat over $B$.

It follows from (\ref{AJ-add}) that 
$\AJ_{\GG}^p = -\AJ_{\FF}^p$.
Since homological dimension of $\GG$ has dropped by one,
induction hypothesis yields the formula (\ref{eqn_dAJ_general}) for $\GG$. It remains to relate right hand side
of (\ref{eqn_dAJ_general}) for $\GG_0$ and for $\FF_0$.

Using functoriality of the Kodaira-Spencer classes 
we obtain the following morphism of triangles:
$$
\begin{tikzcd}[]
\GG_0 \dar{u'}\rar & H^0(X,\FF_0(k))\otimes\OO_X(-k) \dar{0}\rar & \FF_0 \dar{u}\rar & \GG_0[1] \dar{u'[1]} \\
\GG_0[1]\rar       & H^0(X,\FF_0(k))\otimes\OO_X(-k)[1]\rar      & \FF_0[1]\rar      & \GG_0[2]
\end{tikzcd}
$$
where $u = \mathrm{KS}_{\FF_0}(v)\in \Ext^1(\FF_0,\FF_0)$ and $u' = \mathrm{KS}_{\GG_0}(v)\in \Ext^1(\GG_0,\GG_0)$.
Composing the vertical arrows with $\Ati{\FF_0}^{p-1}$, $\Ati{\OO_X(-k)}^{p-1}$ and $\Ati{\FF_0}^{p-1}$ respectively
and using the additivity of traces we get 
$\Tr(\Ati{\FF_0}^{p-1}\circ \mathrm{KS}_{\FF_0}(v)) = -\Tr(\Ati{\GG_0}^{p-1}\circ \mathrm{KS}_{\GG_0}(v))$
because the map in the middle is zero. This completes the induction step.
\end{proof}

\end{document}